\documentclass[reqno]{amsart}
\usepackage{amssymb}

\DeclareMathOperator\C{\mathbb C}
\DeclareMathOperator\Z{\mathbb Z}

\newtheorem{theorem}{Theorem}[section]
\newtheorem{lemma}[theorem]{Lemma}
\newtheorem{cor}[theorem]{Corollary}
\newtheorem{prop}[theorem]{Proposition}
\theoremstyle{definition}
\newtheorem{definition}[theorem]{Definition}

\theoremstyle{remark}
\newtheorem{remark}[theorem]{Remark}

\newcommand{\dontprint}[1]\relax

\newcommand{\we}{\wedge}

\renewcommand{\P}{{\mathbb P}}

\newcommand{\wt}{\widetilde}
\newcommand{\ot}{\otimes}

\renewcommand{\SS}{{\mathcal S}}

\newcommand{\LL}{{\mathcal L}}

\newcommand{\PP}{{\mathcal P}}

\newcommand{\si}{\sigma}

\newcommand{\sub}{\subset}

\newcommand{\ov}{\overline}

\newcommand{\diag}{\operatorname{diag}}

\newcommand{\la}{\lambda}

\newcommand{\GL}{\operatorname{GL}}
\newcommand{\G}{{\mathbb G}}

\newcommand{\lan}{\langle}
\newcommand{\ran}{\rangle}

\newcommand{\codim}{{\operatorname{codim}}}
\renewcommand{\k}{\mathbf{k}}
\newcommand{\val}{{\operatorname{val}}}

\newcommand{\srk}{{\operatorname{srk}}}
\newcommand{\sspan}{{\operatorname{span}}}

\usepackage{xcolor}
\usepackage{amscd}
\title{Linear subspaces of minimal codimension in hypersurfaces}
\author{David Kazhdan}
\author{Alexander Polishchuk}
\thanks{A.P. is partially supported by the NSF grant DMS-2001224, 
and within the framework of the HSE University Basic Research Program and by the Russian Academic Excellence Project `5-100'.}

\address{Einstein Institute of Mathematics,
The Hebrew University of Jerusalem,
Jerusalem 91904, Israel}
\email{kazhdan@math.huji.ac.il}
\address{
    Department of Mathematics, 
    University of Oregon, 
    Eugene, OR 97403, USA; National Research University Higher School of Economics; and Korea Institute for 
    Advanced Study 
  }
  \email{apolish@uoregon.edu}

\begin{document}
\begin{abstract}Let $\k$ be a perfect field and let $X\subset \P^N$ be a 
 hypersurface of degree $d$ defined over $\k$ and containing a linear subspace $L$
defined over $\ov{\k}$ with $\codim_{\P^N}L=r$. We show that  $X$ contains a linear subspace $L_0$ defined over $\k$ with $\codim_{\P^N}L\le dr$.
We conjecture that the intersection of all linear subspaces (over $\ov{\k}$) of minimal codimension $r$ contained in $X$, has codimension bounded above only in
terms of $r$ and $d$. We prove this when either $d\le 3$ or $r\le 2$.
\end{abstract}

\maketitle

\section{Introduction}

Let $f(x_1,\ldots,x_n)$ be a homogeneous polynomial of degree $d\ge 2$ over a field $\k$.
Recall that the {\it slice rank} $\srk_\k(f)$ of $f$ is the minimal number $r$ such that there exists a decomposition
$$f=l_1f_1+\ldots+l_rf_r,$$
where $l_i$ are linear forms defined over $\k$. 

The slice rank $\srk_{\k}(f)$ has a simple geometric meaning: it is the minimal codimension in $\P^{n-1}$ 
of a linear $\k$-subspace $P\sub \P^{n-1}$ contained in
the projective hypersurface $f=0$. Note that if $\srk_{\ov{\k}}(f)<n/2$ then this hypersurface is necessarily singular.

It is clear that $\srk_\k(f)\ge \srk_{\ov{\k}}(f)$ and it is easy to find examples when  $\srk_\k(f)> \srk_{\ov{\k}}(f)$. 

One can ask for an upper estimate for $\srk_\k(f)$ in terms of $\srk_{\ov{\k}}(f)$. 
The first result of this paper is presicely such an estimate, which we obtain
by adapting to homogeneous polynomials the theory of $G$-rank of tensors from the work of Derksen \cite{Derksen}.

A bit more generally, for a collection $f_1,\ldots,f_s$ of homogeneous polynomials of the same degree $d$,
we set
$$\srk_\k(f_1,\ldots,f_s)=\inf_{(c_1,\ldots,c_s)\neq (0,\ldots,0)}\srk_\k(c_1f_1+\ldots+c_sf_s).$$

\medskip

\noindent
{\bf Theorem A}. {\it Assume that the field $\k$ is perfect. Then for a homogeneous polynomial $f$ over $\k$ of degree $d\ge 2$, one has
$$\srk_\k(f)\le d\cdot \srk_{\ov{\k}}(f),$$
$$\srk_\k(f_1,\ldots,f_s)\le d\cdot s\cdot \srk_{\ov{\k}}(f_1,\ldots,f_s).$$
}

\medskip

The inequality for a single $f$ is sharp for every degree $d$: if $E/\k$ is a Galois extension of degree $d$ then the norm $E\to \k$
is a polynomial of degree $d$ that has slice rank $d$ over $\k$ and slice rank $1$ over $E$.

Note that Derksen's theory of $G$-rank for not necessarily symmetric tensors, applied to symmetric tensors implies
only the inequality $\srk_\k(f)\le \frac{d^2}{2}\cdot\srk_{\ov{\k}}(f)$.


Our second goal in this paper is to understand the inequality of Theorem A more constructively. Geometrically, the statement
is that if a hypersurface $X\sub \P^N$ of degree $d$, defined over $\k$, contains a linear subspace $L$ of codimension $r$ in $\P^N$, defined
over $\ov{\k}$, then $X$ contains a linear subspace $L_0$ of codimension $\le dr$ in $\P^N$, defined over $\k$.
One can ask for an explicit construction of  $L_0$ from $L$ and its Galois conjugates. 
The simplest answer would be that one can just take $L_0$ to be the intersection of all the Galois conjugates of $L$. 
We conjecture that the following stronger geometric statement holds.
 
\medskip

\noindent
{\bf Conjecture B}. {\it Let $X=(f=0)\sub \P^N$ be a hypersurface of degree $d$ (over any ground field), and let
 $r$ be the minimal natural number such that $X$ contains
a linear subspace $L$ with $\codim_{\P^N} L=r$. Set 
$$L_f:=\cap_{L\sub X, \codim_{\P^n}L=r} L.$$
Then there exists a function $c(r,d)$ such that $\codim L_f\le c(r,d)$.
}

\medskip

It is an easy exercise to check that the conjecture holds if $d=2$ or $r=1$ with $c(r,2)=2r$ and $c(1,d)=d$.
We prove the following particular cases of the weak form of the conjecture: when $d=3$ (and $r$ is arbitrary) and when $r=2$ (and $d$ is arbitrary).


\medskip

\noindent
{\bf Theorem C}. {\it 
(i) Conjecture B holds for cubic hypersurfaces with 
$$c(r,3)=c(r):=\frac{1}{2}\bigl(\frac{(r+1)^2}{4}+r+3\bigr)\cdot\bigl(\frac{(r+1)^2}{4}+r\bigr).$$

\noindent
(ii) Conjecture B holds for $r=2$ with 
$$c(2,d)=d^2+1.$$
More precisely, for a polynomial $f$ of slice rank $2$ and degree $d$, either $\codim L_f\le d^2-1$ or $f$ is a pullback from a space of dimension
$\le d^2+1$.
}

\medskip

One can ask how far are the estimates of Theorem C from being optimal.
In the case $d=3$ and $r=2$ we show that Conjecture B holds with $c(2,3)=6$ by giving a partial classification of cubic hypersurfaces of rank $2$ (see Theorem \ref{rk2-intersection-thm}).
Consider the cubic $f(x_i,y_{jk})$, where $i=1,\ldots,n$, $1\le j<k\le n$ (so the number of variables is $n(n+1)/2$), given by
$$f=\sum_{i<j}x_ix_jy_{ij}.$$
It is reasonable to expect that the rank of $f$ is equal to $n-1$. Since for every $i<j$, we have $f\in (y_{ij}; x_k \ | \ k\neq i, k\neq j)$, this would imply that $\codim L_f=n(n+1)/2$,
which depends quadratically on the rank. So it seems that the optimal bound $c(r,3)$ is at least quadratic in $r$. 

On the other hand, let us consider a polynomial $f$ in $n$ groups of variables 
$$(x_1(1),\ldots,x_m(1)), \ \ldots, (x_1(n),\ldots,x_m(n))$$
given by
$$f=\sum_{i=1}^m x_1(1)\ldots \widehat{x_i(1)}\ldots x_m(1)\cdot x_i(2)\ldots x_i(n).$$
Then $\deg(f)=m+n-2$ and it is easy to see that $f$ has rank $2$ and $\codim L_f=mn$. This shows that the optimal bound $c(2,d)$ grows quadratically in $d$.

Theorem C(i) implies that if $L$ is a linear subspace of minimal codimension $r$ in $\P^N$, defined over a Galois extension of $\k$, contained in a cubic hypersurface $X$ (defined over $\k$), then
by taking intersection of all Galois conjugates of $L$ we get a linear subspace $L_0$ of codimension $\le c(r)$, contained in $X$ and defined over $\k$. 

Since we don't know the validity of Conjecture B for general $r$ and $d$, 
we give a more complicated construction of constructing a linear subspace $L_0\sub X$, defined over $\k$, starting from the Galois conjugates of
$L\sub X$ defined over a Galois extension of $\k$.
For this we introduce the following recursive definition,
where for linear subspaces $L_1,\ldots,L_s\sub \P^N$ we denote by $\lan L_1,\ldots,L_s\ran\sub \P^N$ their linear span.

\begin{definition}\label{recursion-subspaces-def}
For a collection $\LL=\{L_1,\ldots,L_s\}$ of linear subspaces of $\P^N$, we define a new collection of linear subspaces of $\P^N$
as follows. Let $L=\lan L_1,\ldots,L_s\ran$. For each minimal subset $J\sub [1,s]$ such that $\lan L_j \ |\ j\in J\ran=L$, we set $L_J:=\cap_{j\in J}L_j$, and we denote
by $\LL^{(1)}$ the collection of all such subspaces $L_J$. We denote by $\LL^{(i)}$, $i\ge 1$, the collections of linear subspaces obtained by
iterating this construction. 
\end{definition}

\medskip

\noindent
{\bf Theorem D}. {\it 
Let $X\sub \P^N$ be a hypersurface of degree $d\ge 2$ and let $\LL=(L_1,\ldots,L_s)$ be a collection of linear subspaces
contained in $X$, such that $\codim_{\P^N}L_i\le r$, where $r\ge 2$. Then for the linear subspace
$$L_0:=\lan L \ |\ L\in \LL^{(d-1)}\ran,$$
we have $L_0\sub X$ and 
$$\codim_{\P^N} L_0 \le r^{2^{d-1}}.$$
}

\medskip

Applying the construction of Theorem D to the collection of all Galois conjugates of a linear subspace of codimension $r\ge 2$ in $\P^N$,
defined over some Galois extension of $\k$, contained in a hypersurface $X$ (defined over $\k$), we get an algorithm for producing a linear subspace of codimension 
$\le r^{2^{d-1}}$ in $\P^N$, contained in $X$ and defined over $\k$.

One can ask whether the second inequality in Theorem A can also be explained constructively. In other words,
starting with an $s$-dimensional subspace $F$ of homogeneous polynomials of the same degree $d$, defined over a perfect field $\k$, such that there exists a nonzero
$f\in F_{\ov{\k}}$ and subspace of linear forms of dimension $r$ over $\ov{\k}$ such that $f\in (L)$, we want to produce an element $f_0\in F\setminus 0$ and a subspace
of linear forms $L_0$, both defined over $\k$, such that $f_0\in (L_0)$ and dimension of $L_0$ is $\le c(sr)$.
In Remark \ref{family-rk-rem} we show how to do this using the algorithm of Theorem D for a single polynomial.



Our study is partially motivated by the desire to understand 
the related notion of the {\it Schmidt rank} (also known as {\it strength}) of a homogeneous polynomial (see \cite{AKZ}, \cite{BBOV} and references therein), 
defined as the minimal number $r$ such that $f$ admits a decomposition $f=g_1h_1+\ldots+g_rh_r$, with $\deg(g_i)$ and $\deg(h_i)$ smaller
than $\deg(f)$. Similarly to Theorem A one can try to estimate the Schmidt rank of a polynomial
over a non-closed field in terms of its Schmidt rank over an algebraic closure. In \cite{KP-quartics}, we show how to do this for quartic polynomials. 

\bigskip

\noindent
{\it Acknowledment}. The second author is grateful to Nick Addington for introducing him to Macaulay 2, 
which helped to find some examples of polynomials of small rank with large $\codim L_f$. 

\section{$G$-rank for homogeneous polynomials}

Throughout this section we assume that the ground field $\k$ is perfect.








\subsection{Definition of the $G$-rank and the relation to the slice rank}

Below we introduce an analog of $G$-rank for symmetric tensors, or equivalently, for homogeneous polynomials,
$r_\k^G(f)$ (where $G=\GL_n$). We show that it enjoys similar properties to the Derksen's $G$-rank of a non-symmetric tensor
studied in \cite{Derksen}, in particular, it does not change under algebraic extensions of perfect fields. 
We also introduce the notion $r^G_\k(f_1,\ldots,f_s)$ of a $G$-rank for a collection of polynomials of the same degree.



Let $V$ be an $n$-dimensional space over $\k$. We consider the group $G=\GL(V)\simeq \GL_n(\k)$ 
acting naturally on the space $S^dV$, and the induced action on ${\bigwedge}^s(S^dV)$. 

We consider points of $G$ and of ${\bigwedge}^s S^dV$ with values in the ring of formal power series $\k[\![t]\!]$.
For a $\k$-vector space $W$ and a vector $w\in W[\![t]\!]$, we denote by $\val_t(w)$ the minimal $m\ge 0$ such that
$w\in t^m W[\![t]\!]$.

For $f\in S^dV$ and $g(t)\in G(\k[\![t]\!])$ such that $\val_t(g(t)\cdot f)>0$, we set
$$\mu(g(t),f)=d\cdot \frac{\val_t(\det(g(t)))}{\val_t(g(t)\cdot f)}.$$
The factor $d$ in front is a matter of convention: it makes the factor $d$ disappear in some of the statements below. 

\begin{definition} (i) For nonzero $f\in S^dV$ we define its $G$-rank by 
$$r^G_\k(f)=\inf_{g(t)}\mu(g(t),f),$$
where we take the infimum over all $g(t)\in G(\k[\![t]\!])$ such that $\val_t(g(t)\cdot f)>0$.

\noindent
(ii) More generally, for linearly independent $f_1,\ldots,f_s\in S^dV$, we define the $G$-rank by
$$r^G_\k(f_1,\ldots,f_s)=\inf_{g(t)}\mu(g(t),f_1,\ldots,f_s),$$
where
$$\mu(g(t),f_1,\ldots,f_s)=ds\cdot \frac{\val_t(\det(g(t)))}{\val_t(g(t)\cdot f_1\we\ldots\we f_s)},$$
and the infinum is taken over $g(t)\in G(\k[\![t]\!])$ such that $\val_t(g(t)\cdot f_1\we\ldots\we f_s)>0$.
\end{definition}

The formula 
$$\val_t(g(t)\cdot f^m)=m\cdot \val_t(g(t)\cdot f)$$
immediately implies the following property.

\begin{lemma}\label{power-lem}
For any $f\in S^d V$ and any $m\ge 1$ one has
$$r^G_{\k}(f^m)=r^G_{\k}(f).$$
\end{lemma}

Here is the main result connecting the $G$-rank with the slice rank and also with the Waring rank.

\begin{theorem}\label{Grk-thm} 
Assume the base field $\k$ is perfect.

\noindent
(i) For a homogeneous polynomial $f$ of degree $d$ over $\k$ one has
$$\srk_\k(f)\le r_\k^G(f)\le d\cdot \srk_\k(f).$$
For a collection $f_1,\ldots f_s$ of homogeneous polynomials of degree $d$ over $\k$ one has
$$\srk_\k(f_1,\ldots,f_s)\le r_\k^G(f_1,\ldots,f_s)\le ds\cdot \srk_\k(f_1,\ldots,f_s).$$

\noindent
(ii) Suppose $m_1,\ldots,m_r$ are divisors of $d$, and $f_1,\ldots,f_r$ are homogeneous polynomials of degrees $\deg(f_i)=d/m_i$. Then
$$r_\k^G(f_1^{m_1}+\ldots+f_r^{m_r})\le r_\k^G(f_1)+\ldots+r_\k^G(f_r).$$
In particular, 
$$r_\k^G(f)\le w_\k(f),$$
where $w_\k(f)$ is the Waring rank of $f$, i.e., the minimal number $r$ such that 
$$f=l_1^d+\ldots+l_r^d,$$
where $l_i$ are linear forms defined over $\k$.
\end{theorem}

The proof will be given in Sec.\ \ref{Grank-slice-rank-proof-sec} after some preparations.
The argument is very close to the one in \cite{Derksen}.

\subsubsection{Relation to the GIT stability}

Let $W$ be a finite dimensional algebraic representation of $G=\GL(V)$ over $\k$. Recall that a point $w\in W$ is called
{\it $G$-semistable} if the orbit closure $\ov{G\cdot v}$ does not contain $0$. 
Recall that Kempf's $\k$-rational version of the Hilbert-Mumford criterion (see \cite{Kempf}) 
states (assuming $\k$ is perfect) that 
if $w$ is not $G$-semistable then there exists a $1$-parameter subgroup $\la:\G_m\to G$ {\it defined over $\k$}
such that $\lim_{t\to 0} \la(t)\cdot w=0$. Here $\la$ has form $g\cdot \diag(t^{\la_1},\ldots,t^{\la_n})\cdot g^{-1}$
for some $g\in G(\k)$ and $\la_i\in\Z$.

\begin{prop}\label{GIT-prop} 
For integers $p\ge 0$ and $q>0$, let us consider the $G$-representation
$$W=({\bigwedge}^s S^dV)^{\ot p}\ot {\det}^{-dsq}\oplus V^n.$$
Let $u\in V^n$ be a fixed element of rank $n$. Then we have $r^G_\k(f_1,\ldots,f_s)\ge \frac{p}{q}$ if and only if
$w=((f_1\we\ldots\we f_s)^{\ot p}\ot 1,u)$ is $G$-semistable.
\end{prop}

\begin{proof} By Hilbert-Mumford-Kempf's criterion, if $w$ is not $G$-semistable then there exists a $1$-parameter subgroup
$\la:\G_m\to G$ over $\k$, such that $\lim_{t\to 0}\la(t)\cdot w=(0,0)$. In particular, we have $\lim_{t\to 0}\la(t)\cdot u=0$,
so $\la(t)\in G(\k[t])$, and
$$\val_t(\la(t)\cdot (f_1\we\ldots\we f_s)^{\ot p}\ot 1)=p\cdot \val_t(\la(t)\cdot f_1\we\ldots\we f_s)-dsq\cdot \val_t\det(\la(t))>0,$$
which implies that $\val_t(\la(t)\cdot f_1\we\ldots\we f_s)>0$ and
$$\mu(\la(t),f_1,\ldots,f_s)<\frac{p}{q}.$$
Hence, $r^G(f_1,\ldots,f_s)<\frac{p}{q}$.

Conversely, assume there exists $g(t)\in G(\k[\![t]\!])$ such that $\val_t(g(t)\cdot  f_1\we\ldots\we f_s)>0$ and
$\mu(g(t),f_1,\ldots,f_s)<\frac{p}{q}$, i.e.,
$$\val_t(g(t)\cdot (f_1\we\ldots\we f_s)^{\ot p}\ot 1)>0.$$
Truncating $g(t)$ at high enough order in $t$, we can assume that $g(t)\in G(\k[t])$.
Then the fact that $\lim_{t\to 0} g(t)\cdot w=(0,g(0)\cdot u)$ implies that $(0,g(0)\cdot u)$ lies in the closure of the $G$-orbit of $w$. 
Since $0$ lies in the closure
of the $G$-orbit of $g(0)\cdot u$, we see that $(0,0)$ lies in the closure of $G\cdot w$, so $w$ is not $G$-semistable.
\end{proof}

As a consequence of Proposition \ref{GIT-prop}, in the definition of $r^G_\k(f_1,\ldots,f_s)$ it is enough to take $g(t)$ to be a $1$-parameter
subgroup of $G$ defined over $\k$.
Also, since $G$-semistability does not change under the base field extension, we deduce the following 

\begin{cor}\label{Grk-field-change-cor} 
Let $\ov{\k}$ be an algebraic closure of $\k$.
Then one has 
$$r^G_{\k}(f_1,\ldots,f_s)=r^G_{\ov{\k}}(f_1,\ldots,f_s).$$
\end{cor}

Let $T\sub G$ denote the maximal torus, i.e., the group of diagonal matrices with respect to a $\k$-basis $(e_i)$ of $V$.
Replacing $G$ everywhere by $T$ we get a notion of $T$-rank, $r^T_\k(f_1,\ldots,f_s)$.
From Hilbert-Mumford-Kempf criterion we get
$$r^G_\k(f_1,\ldots,f_s)=\inf_{g\in G(\k)} r^T_\k(g\cdot (f_1,\ldots,f_s)).$$

The reason we introduced the factor $ds$ in the definition of $r^G_\k(f_1,\ldots,f_s)$ is so as to have the following normalization property.

\begin{lemma}\label{ge1-ineq-lem}
One has $r^G_\k(f_1,\ldots,f_s)\ge 1$.
\end{lemma}

\begin{proof} It is enough to check that for any $g(t)\in T(\k[\![t]\!])$ and any distinct monomials $M_1,\ldots,M_s$ of $(e_i)$ in $S^dV$, one has
$$\val_t(g(t)\cdot M_1\we\ldots\we M_s)\le ds\cdot \val_t(\det(g(t))).$$
Let $c_1,\ldots,c_n\ge 0$ be the valuations of the diagonal entries of $g(t)$, so that
$\val_t(\det(g(t)))=c_1+\ldots+c_n$.
Then for a monomial $M=e_1^{a_1}\ldots e_n^{a_n}$, we have
$$\val_t(g(t)\cdot M)=a_1c_1+\ldots+a_nc_n\le (a_1+\ldots+a_n)(c_1+\ldots+c_n)=d(c_1+\ldots+c_n).$$
Hence, $\val_t(g(t)\cdot M_1\we\ldots\we M_s)\le ds$,
which gives the required inequality.
\end{proof}




\subsection{Triangle inequality}

\begin{prop}\label{tr-ineq-prop}
For $f_1,f_2\in S^dV$ one has $r^G_\k(f_1+f_2)\le r^G_\k(f_1)+r^G_\k(f_2)$.
\end{prop}

\begin{proof} This is proved exactly as \cite[Prop.\ 3.6]{Derksen}. Starting with $g_1(t),g_2(t)\in G(\k[\![t]\!])$ such that
$\val_t(g_i(t)\cdot f_i)>0$, one has to produce $u(t)\in G(\k[\![t]\!])$ with $\val_t(u(t)\cdot (f_1+f_2))>0$ and 
$$\mu(u(t),f_1+f_2)\le \mu(g_1(t),f_1)+\mu(g_2(t),f_2).$$
Making changes of variables $t\mapsto t^i$ if necessary, we can assume that
$$\val_t(g_1(t)\cdot f_1)=\val_t(g_2(t)\cdot f_2)=s>0.$$
By \cite[Lem.\ 3.5]{Derksen}, there exists $u(t)\in G(\k[\![t]\!])$ such that
$u(t)=u_1(t)g_1(t)=u_2(t)g_2(t)$ with $u_i(t)\in G(\k[\![t]\!])$ and
$$\val_t(\det u(t))\le \val_t(\det g_1(t))+\val_t(\det g_2(t)).$$
Then
$$\val_t(u(t)\cdot (f_1+f_2))\ge \min(\val_t(u_1(t)g_1(t)\cdot f_1),\val_t(u_2(t)g_2(t)\cdot f_2))\ge s,$$
and
\begin{align*}
&\frac{1}{d}\mu(u(t),f_1+f_2)=\frac{\val_t(\det u(t))}{\val_t(u(t)\cdot (f_1+f_2))}\le \\
&\frac{\val_t(\det g_1(t))+
\val_t(\det g_2(t))}{s}=\frac{1}{d}(\mu(g_1(t),f_1)+\mu(g_2(t),f_2)).
\end{align*}
\end{proof}


\subsection{Relation to the slice rank and to the sums of powers}\label{Grank-slice-rank-proof-sec}


\begin{prop}\label{Grk-srk-prop1} 
(i) Let $f=v^d$ for some $v\in V\setminus 0$. Then $r^G_\k(f)=1$.

\noindent
(ii) If $f$ is of slice rank $r$ then $r^G_\k(f)\le dr$.

\noindent
(iii) If there exists a nontrivial linear combination $c_1f_1+\ldots+c_sf_s$ that has slice rank $r$ then $r^G_\k(f_1,\ldots,f_s)\le dsr$.
\end{prop}

\begin{proof} (i) By Lemma \ref{ge1-ineq-lem}, $r^G_\k(f)\ge 1$, so it is enough to find $g(t)\in G(\k[t])$
such that $\mu(g(t),v^d)=1$. We can assume that $v=e_1$, and take 
$$g(t)=\diag(t,1,\ldots,1).$$
Then $\val_t(g(t)\cdot e_1^d)=d$ and $\val_t(\det(g))=1$.
(Alternatively, we can use Lemma \ref{power-lem} to reduce to the easy case $d=1$.)

\noindent
(ii) We can assume that $f=e_1\cdot f_1+\ldots+e_r\cdot f_r$. Then for 
$g(t)=\diag(\underbrace{t,\ldots,t}_{r},1,\ldots,1)$, we have $\val_t(g(t)\cdot f)\ge 1$, while $\val_t(\deg(g))=r$, so
$$\mu(g(t),f)\le d\cdot \frac{r}{\val_t(g(t)\cdot f)}\le dr.$$

\noindent
(iii) If this is the case then $f_1\we\ldots\we f_s$ has form $(e_1h_1+\ldots+e_rh_r)\we\ldots$, hence, for the same
$g(t)$ as in (ii), we have $\val_t(g(t)\cdot f_1\we\ldots\we f_s)\ge 1$.
\end{proof}


\begin{prop}\label{Grk-srk-prop2} 
One has 
$$
\srk_\k(f_1,\ldots,f_s)\le r^G_\k(f_1,\ldots,f_s).$$
\end{prop}

\begin{proof}
Suppose $r^G_\k(f_1,\ldots,f_s)<r$. Then there exists a $1$-parameter subgroup $g(t)$
such that 
$$r\cdot\val_t(g(t)\cdot f_1\we\ldots\we f_s)> ds\cdot \val_t(\det(g(t))).$$
We can assume that $g(t)$ is diagonal with respect to some basis $(e_1,\ldots,e_n)$ of $V$.
Now consider the set
$$S:=\{ i\in [1,n] \ |\ \val_t(g(t)\cdot e_i)\ge \frac{\val_t(g(t)\cdot f_1\we\ldots\we f_s)}{ds}\}.$$
Note that
$$\val_t(\deg(g(t)))\ge \sum_{i\in S} \val_t(g(t)\cdot e_i)\ge |S|\cdot \frac{\val_t(g(t)\cdot f_1\we\ldots\we f_s)}{ds},$$
hence, 
$$|S|<r.$$

We claim that there exists a nontrivial linear combination $f=c_1f_1+\ldots+c_sf_s$ such that all the monomials
appearing in $f$ are divisible by some $e_i$ with $i\in S$. Indeed, otherwise, the projection
$$\lan f_1,\ldots,f_s\ran\to \k[e_1,\ldots,e_n]\to \k[e_1,\ldots,e_n]/(e_i \ |\ i\in S)\simeq \k[e_i \ | i\not\in S]$$
is injective, so there exist $s$ distinct monomials $M_1,\ldots,M_s$ of degree $d$ in $\k[e_i \ | i\not\in S]$ such that
$M_1\we\ldots\we M_s$ appears with a nonzero coefficient in $f_1\we\ldots\we f_s$. But then by the choice of $S$,
$$\val_t(g(t)\cdot f_1\we\ldots \we f_s)\le \val_t(g(t)\cdot M_1\we\ldots\we M_s)<\val_t(g(t)\cdot f_1\we\ldots\we f_s)$$
which is a contradiction, proving our claim. Now for the obtained linear combination $f$ we have 
$$\srk_\k(f)\le |S|<r.$$
\end{proof}

\medskip

\begin{proof}[Proof of Theorem \ref{Grk-thm}]
(i) This follows from Proposition \ref{Grk-srk-prop1}(iii) and Proposition \ref{Grk-srk-prop2}.

\noindent
(ii) This follows from Proposition \ref{tr-ineq-prop} (the triangle inequality), Lemma \ref{power-lem} and Proposition \ref{Grk-srk-prop1}(i) (for
the part concerning the Waring rank).
\end{proof}

\medskip

\begin{proof}[Proof of Theorem A]
We combine Theorem \ref{Grk-thm}(i) with Corollary \ref{Grk-field-change-cor}:
$$\srk_{\k}(f_1,\ldots,f_s)\le r_\k^G(f_1,\ldots,f_s)=r_{\ov{\k}}^G(f_1,\ldots,f_s)\le ds\cdot \srk_{\ov{\k}}(f_1,\ldots,f_s).$$
\end{proof}

\subsection{Example of a calculation of $G$-rank}

As we have seen before, for any linear form form $l$ one has $r_\k^G(l^d)=1$.
Here is the next simplest case.

\begin{prop} 
Assume that $n=2$. Then for any $m>0$, one has 
$$r_\k^G(x_1^{2m}x_2^m)=\frac{3}{2}.$$
\end{prop}

\begin{proof}
By Lemma \ref{power-lem}, it is enough to prove that
$$r_\k^G(x_1^2x_2)=\frac{3}{2}.$$
Considering $g(t)=\diag(t,1)$, we immediately see that $r_\k^G(x_1^2x_2)\le 3/2$.

Now consider any 
$$g=\left(\begin{matrix} a & b \\ c & d\end{matrix}\right)\in G(\k[\![t]\!]).$$
It is enough to prove that $\mu(g,x_1^2x_2)\le 3/2$.
We have 
$$g\cdot x_1^2x_2=a^2c\cdot x_1^3+a(ad+2bc)\cdot x_1^2x_2+b(bc+2ad)\cdot x_1x_2^2+b^2d\cdot x_2^3.$$
Let us abbreviate $v(\cdot)=\val_t(\cdot)$, etc.
Set $s:=v(g\cdot x_1^2x_2)$.
Then we have 
$$2v(a)+v(c)\ge s, \ v(a)+v(ad+2bc)\ge s, \ v(b)+v(bc+2ad)\ge s, \ 2v(b)+v(d)\ge s.$$

We consider three cases.

\noindent
{\bf Case $v(ad)>v(bc)$.}

Then we have $v(\det(g))=v(bc)$ and $v(bc+2ad)=v(bc)$. Hence, from the above inequalities we get
$v(b)+v(bc)\ge s$, hence, $2v(bc)\ge s$, so $v(\det(g))=v(bc)\ge s/2$, and so $\mu(g,x_1^2x_2)\ge 3/2$.

\noindent
{\bf Case $v(ad)<v(bc)$.}

Then we have $v(\det(g))=v(ad)$ and $v(ad+2bc)=v(ad)$. Hence,
$2v(ad)\ge v(a)+v(ad)\ge s$, and we again get $v(\det(g))\ge s/2$. 

\noindent
{\bf Case $v(ad)=v(bc)$.}

Set $t=v(ad)=v(bc)$. Then we have $v(\det(g))\ge t$. Now by the above inequalities,
$$4t=2v(ad)+2v(bc)\ge (2v(a)+v(c))+(2v(b)+v(d))\ge 2s,$$
which again implies $v(\det(g))\ge s/2$.
\end{proof}

\section{Linear subspaces of minimal codimension in cubics}

\subsection{Some general observations}

Let $f\in \k[V]$ be a nonzero homogeneous polynomial of slice rank $r$, and let $X\sub\P V$ be the corresponding projective hypersurface. We are interested in the intersection 
$$L_f:=\cap_{L\sub X, \codim_{\P V}L=r} L \sub \P V.$$
Recall that we are looking for an estimate for the codimension of $L_f$.
The case $r=1$ is straightforward:

\begin{lemma}\label{rk1-lem} 
Let $f$ be a homogeneous polynomial of degree $d$ and slice rank $1$. Then there are most $d$ hyperplanes contained in $X$, so $\codim_{\P V} L_f\le 3$.
\end{lemma}

Since the slice rank is determined in terms of ideals $(P)\sub \k[V]$ generated by subspaces $P$ of linear forms, we record some easy observations about such ideals.

\begin{lemma}\label{flat-lem}
Let $A\sub B$ be an extension of commutative rings, such that $B$ is flat as $A$-algebra. Then for any pair of ideals $J_1,J_2\sub A$, one has
$$(J_1\cdot B)\cap (J_2\cdot B)=(J_1\cap J_2)\cdot B.$$
In particular, for a collection of linear subspaces $P_i\sub W$, $i=1,\ldots,s$, where $W\sub V^*$ is a subspace, we have
$$P_1\k[V]\cap\ldots P_s\k[V]=(P_1S(W)\cap\ldots\cap P_sS(W))\cdot \k[V].$$
\end{lemma}

\begin{proof} Since for any ideal $J\sub A$ the natural map $J\ot_A B\to J\cdot B$ is an isomorphism in this case, the assertion follows by applying the exact functor $?\ot_A B$ to the exact sequence
$$0\to J_1\cap J_2\to J_1\oplus J_2\to A.$$
For the last statement we apply this to the flat extension of rings $S(W)\sub S(V^*)=\k[V]$.
\end{proof}

\begin{lemma}\label{ideals-intersection-lem}
Let $P_1,\ldots,P_s\sub V^*$ be subspaces such that the ideal $(P_1)^{a_1}\cap \ldots \cap (P_s)^{a_s}$ contains no nonzero homogeneous polynomials of degree $m$, for
some powers $a_i\ge 1$.
Then we have an inclusion of ideals in $\k[V]$,
$$(P_1)^{a_1}\cap\ldots\cap(P_s)^{a_s}\sub (W)^{m+1}.$$
where $W=P_1+\ldots+P_s$.
In particular, if $P_1\cap\ldots \cap P_s=0$ then  
$$(P_1)^m\cap\ldots\cap(P_s)^m\sub (W)^{m+1}.$$
\end{lemma}

\begin{proof}
Applying Lemma \ref{flat-lem} to the extension of rings $S(W)\sub S(V^*)=\k[V]$, we reduce to the case when $W=V^*$. 
But then the first statement reduces to the fact that if the ideal $I=(P_1)^{a_1}\cap\ldots\cap(P_s)^{a_s}$ does not contain polynomials of degree $m$
then $I\sub (x_1,\ldots,x_n)^{m+1}$.

To prove the second statement we need to check that $(P_1)^m\cap \ldots \cap (P_s)^m$ does not contain any homogeneous polynomials of degree $\le m$.
This is clear in degrees $<m$ and in degree $m$ follows from the statement that
$$0=S^m(P_1\cap\ldots\cap P_s)=S^m(P_1)\cap\ldots\cap S^m(P_s)\sub S^mW,$$
since $P_1\cap\ldots\cap P_s=0$.
\end{proof}

We say that a polynomial $f\in \k[V]=S(V^*)$ is a {\it pullback from a space of dimension $m$} if there exists a linear subspace $W\sub V^*$ of
dimension $m$ such that $f\in S(W)\sub S(V^*)$.
In this case, if $f\in (P)$, where $P\sub V^*$ is a subspace of linear forms, then $f\in (W\cap P)$. In particular, the slice rank of $f$ in $S(V^*)$ can be calculated 
within $S(W)$.

\subsection{Proof of Theorem C(i)}

Theorem C(i) is a consequence of the following more precise theorem.

\begin{theorem}\label{cubics-rough-bound-thm} Let $f$ be a cubic of rank $r$, $X\sub \P V$ the corresponding hypersurface. Set
$$c(r):=\frac{1}{2}\bigl(\frac{(r+1)^2}{4}+r+3\bigr)\cdot\bigl(\frac{(r+1)^2}{4}+r\bigr).$$
Then 
\begin{itemize}
\item either all linear subspaces $L\sub X$ with $\codim_{\P V}L=r$ are contained in a fixed hyperplane,
\item or $f$ is a pullback from a space of dimension $c(r)$.
\end{itemize}
In either case $\codim L_f \le c(r)$.
\end{theorem}


\begin{lemma}\label{subspaces-intersection-lem}
Let $P_1,\ldots,P_s\sub V^*$ be an irredundant collection of subspaces such that $P_1\cap\ldots \cap P_s=0$
(i.e., the intersection of any proper subcollection is nonzero). 
Assume that $\dim P_i\le r$ for every $i$. Then 
$$\dim(P_1+\ldots+P_s)\le r+\frac{(r+1)^2}{4}.$$
\end{lemma}

\begin{proof}
Let $a$ be the minimal dimension of intersections $P_i\cap P_j$. Then we claim that $s\le a+2$. Indeed, without loss of generality
we can assume that $\dim P_1\cap P_2=a$. Then for each $i\ge 2$ we should have 
$$\dim P_1\cap P_2\cap\ldots \cap P_i\le a+2-i,$$
due to irredundancy of the collection, which proves the claim for $i=s$.

On the other hand, since $\dim P_i/(P_i\cap P_1)\le r-a$ for $i>1$, we get that
$$N:=\dim(P_1+\ldots+P_s)\le r+(s-1)(r-a)\le r+(a+1)(r-a)\le r+\frac{(r+1)^2}{4}.$$
\end{proof}

\medskip

\begin{proof}[Proof of Theorem \ref{cubics-rough-bound-thm}]
We use induction on $r$. For $r=1$ the assertion is clear.
Assume $r>1$ and the assertion holds for $r-1$.
Let $\PP_f$ denote the set of $r$-dimensional subspaces $P\sub V^*$ such that $f|_{P^\perp}=0$, or equivalently, $f\sub (P)$.

If all $P\in \PP_f$ contain the same line $(v^*)$ then we can apply the induction assumption to the restriction of $f$ to the hyperplane $v^*=0$ in $V$,
which has slice rank $r-1$. Then the induction assumption implies that
$$\codim_{\P V} L_f\le c(r-1)+1\le c(r).$$

Otherwise, there exist $P_1,\ldots,P_s\in \PP_f$ such that $P_1\cap\ldots\cap P_s=0$.
Choosing a minimal such collection of subspaces and using Lemma \ref{subspaces-intersection-lem}, we get
$$N:=\dim(P_1+\ldots+P_s)\le r+\frac{(r+1)^2}{4}.$$
Now by Lemma \ref{ideals-intersection-lem}, $f$ belongs to $(W)\cdot (W)$, where $W=P_1+\ldots+P_s$.
Hence, $f$ can be written in the form
$$f=\sum_{1\le i\le j\le N}w_iw_jl_{ij},$$
for some linear forms $l_{ij}$, where $(w_i)$ is a basis of $W$. Hence, $f$
 is a pullback from a space of dimension $\le \frac{N(N+1)}{2}+N\le c(r)$. 
\end{proof}

\subsection{Cubics of slice rank $2$}

Here we study in more detail the case of cubics of slice rank $2$, proving in this case Conjecture B with $c(2,3)=6$ and partially classifying such cubics.

\begin{theorem}\label{rk2-intersection-thm} Let $f$ be a cubic of rank $2$. Then 
\begin{itemize}
\item either all $L\sub X$ with $\codim_{\P V}=2$ are contained in a fixed hyperplane, or 
\item $f$ is a pullback from a $6$-dimensional space, or 
\item $f$ can be written in the form $f=x_1y_1z_1+x_1y_2z_2+x_2y_1z_3$, where $x_1,x_2,y_1,y_2,z_1,z_2,z_3$ are linearly independent, or 
\item $f$ is a pullback from an $8$-dimensional space and $\codim_{\P V} L_f\le 4$, or 
\item $f$ is a pullback from a $9$-dimensional space and $\codim_{\P V} L_f\le 3$.
\end{itemize}
In either case $\codim_{\P V} L_f\le 6$.
\end{theorem}

From now on we fix a cubic $f\in \k[V]$ of slice rank $2$. 
As in the proof of Theorem \ref{cubics-rough-bound-thm} we denote by
$\PP_f$ the set of $2$-dimensional subspaces $P\sub V^*$ such that $f|_{P^\perp}=0$, or equivalently, $f\sub (P)$,
where $(P)\sub \k[V]$ denotes the ideal generated by $P$.

The following result is well known but we include the (simple) proof for reader's convenience.

\begin{lemma}\label{proj-lines-lem} 
Let $\SS$ be a set of $2$-dimensional subspaces in $V^*$ such that for any $P_1,P_2\in \SS$ we have $P_1\cap P_2\neq 0$.
Then either there exists a line $L\sub V^*$ such that $L\sub P$ for all $P\in \SS$, or there exists a $3$-dimensional subspace $W\sub V^*$ such that
$P\sub W$ for all $P\in \SS$.
\end{lemma}

\begin{proof} We can think of $\SS$ as a family of projective lines in the projective space such that any two intersect. Our statement is that either they all pass through
one point, or they are contained in a plane. Indeed, assume they do not all pass through one point. Pick a pair of lines $\ell_1,\ell_2$ intersecting at a point $p$.
There exists a line $\ell_3$, not passing through $p$. Then $\ell_1,\ell_2,\ell_3$ form a triangle in a plane. Now given any other line $\ell$ from $\SS$, we can pick a vertex
of the triangle such that $\ell$ does not pass through it. Say, assume $\ell$ does not pass through $p$. Then $\ell\cap \ell_1$ and $\ell\cap \ell_2$ are two distinct points of $\ell$,
so $\ell$ is contained in the plane of the triangle.
\end{proof}

\begin{lemma}\label{nontriv-int-case} 
Assume that for any pair $P_1,P_2\in \PP_f$ we have $P_1\cap P_2\neq 0$. Then either there exists a nonzero linear form $v^*\in V^*$, such that
$v^*\in P$ for all $P\in \PP_f$, in which case $\codim_{\P V}L_f\le 4$, or $f$ is a pullback from a $9$-dimensional space and $\codim_{\P V}L_f\le 3$.
\end{lemma}

\begin{proof} By Lemma \ref{proj-lines-lem}, either all planes in $\PP_f$ span at most $3$-dimensional subspace $W\sub V^*$, or
there exists a nonzero linear form $v^*\in V^*$ such that $v^*\in P$ for all $P\in \PP_f$. In the latter case
let us consider the restriction $\wt{f}$ of our cubic
to the hyperplane $H_{v^*}\sub V$. Then $\wt{f}$ has rank $1$ and $\PP_f$ can be identified with $\PP_{\wt{f}}$. So by Lemma \ref{rk1-lem},
$L_f$ has codimension $3$ in $H_{v^*}$,
hence, it has codimension $4$ in $V$.

Now let us consider the case when all planes in $\PP_f$ are contained in a $3$-dimensional subspace $W$, and have zero intersection.
Then by Lemma \ref{ideals-intersection-lem}, $f\in (W)^2$. Hence, 
as in the proof of Theorem \ref{cubics-rough-bound-thm}, we deduce that $f$ depends on $\le 9$ variables. 
\end{proof}

\begin{lemma}\label{main-rk2-cubics-lem} 
Assume there exists linearly independent linear forms $x_1,x_2,y_1,y_2\in V^*$ such that $\sspan(x_1,x_2)\in \PP_f$ and $\sspan(y_1,y_2)\in \PP_f$.
Then $f$ is a pullback from an $8$-dimensional space, and one of the following possibilities hold:
\begin{enumerate}
\item $f$ is a pullback from a $6$-dimensional space;
\item for all $P\in \PP_f$ one has $P\sub \sspan(x_1,x_2,y_1,y_2)$;
\item $f$ can be written in the form $f=x_1y_1z_1+x_1y_2z_2+x_2y_1z_3$, where $x_1,x_2,y_1,y_2,z_1,z_2,z_3$ are linearly independent.
\end{enumerate}
\end{lemma}

\begin{proof}
Note that we can write
$$f=x_1y_1l_{11}+x_1y_2l_{12}+x_2y_1l_{21}+x_2y_2l_{22},$$
for some linear forms $l_{ij}\in V^*$. This immediately implies that $f$ depends on $\le 8$ variables.

Let $P=\sspan(l_1,l_2)$ be in $\PP_f$. First, we claim that if $P\cap \sspan(x_1,x_2)=0$ and $P\cap \sspan(y_1,y_2)=0$ then either $P\sub \sspan(x_1,x_2,y_1,y_2)$
or $f$ is a pullback from a $6$-dimensional space.
Indeed, assume that $P$ is not contained in $\sspan(x_1,x_2,y_1,y_2)$. First, we observe that 
for generic $x\in \sspan(x_1,x_2)$ and generic $y\in \sspan(y_1,y_2)$ we should have
$P\cap \sspan(x,y_1,y_2)=0$ and $P\cap \sspan(y,x_1,x_2)=0$.  Indeed, otherwise we could pick generic $x,x'\in \sspan(x_1,x_2)$ such that there exist nonzero vectors
$v\in P\cap \sspan(x,y_1,y_2)$ and $v'\in P\cap \sspan(x',y_1,y_2)$. But then, since $P\cap \sspan(y_1,y_2)=0$, we would have that $v$ and $v'$ are linearly independent,
and so $P=\sspan(v,v')\sub (x_1,x_2,y_1,y_2)$.
Hence, changing bases $\sspan(x_1,x_2)$ and $\sspan(y_1,y_2)$ if necessary, we can assume that
$$P\cap \sspan(x_1,y_1,y_2)=P\cap \sspan(x_2,y_1,y_2)=P\cap \sspan (y_1,x_1,x_2)=P\cap \sspan(y_2,x_1,x_2)=0.$$
Now the fact that $f\in (P)$ implies that
$$x_1(y_1l_{11}+y_2l_{12})\in (x_2,P).$$
Since $x_1\not\in (x_2,P)$, we get that $y_1l_{11}+y_2l_{12}\in (x_2,P)$. Hence,
$$y_1l_{11}\in (x_2,y_2,P).$$
We know that $y_1\not\in (x_2,y_2,P)$ since otherwise we would get a nonzero intersection $P\cap (x_2,y_1,y_2)$. Hence $l_{11}\in (x_2,y_2,P)$.
Similarly, we get $l_{12}\in (x_2,y_1,P)$, $l_{21}\in (x_1,y_2,P)$, and $l_{22}\in (x_1,y_1,P)$. But this implies that $f$ is a pull-back from a $6$-dimensional space.

It remains to consider the case when there exists $P$ in $\PP_f$, such that $P\cap \sspan(y_1,y_2)=0$ and $P\cap \sspan(x_1,x_2)=\sspan(x_1)$.
Then the condition $f\in (P)$ gives
$$x_2(y_1l_{21}+y_2l_{22})\in (P).$$
Hence, $y_1l_{21}+y_2l_{22}\in (P)$, which implies that
$$y_1l_{21}\in (y_2,P).$$
Since $y_1\not\in (y_2,P)$, we get $l_{21}\in (y_2,P)$. Similarly, we get $l_{22}\in (y_1,P)$.
Let $P=\sspan(x_2,l)$, where $l\in V^*$. Then we can write 
$$l_{21}=a_1x_1+b_1y_2+c_1l, \ \ l_{22}=a_2x_1+b_2y_1+c_2l,$$
so we can rewrite $f$ in the form
$$f=x_1y_1(l_{11}+a_1x_2)+x_1y_2(l_{12}+a_2x_2)+x_2(c_1y_1+c_2y_2)l+(b_1+b_2)x_2y_1y_2.$$
The condition $f\in (x_1,l)$ gives $(b_1+b_2)x_2y_1y_2\in (x_1,l)$, which is possible only if $b_1+b_2=0$.
This easily implies that either $f$ is a pullback from a $6$-dimensional space, or can be written in the form (3).
\end{proof}

\medskip
\begin{proof}[Proof of Theorem \ref{rk2-intersection-thm}]
Taking into account Lemmas \ref{nontriv-int-case} and \ref{main-rk2-cubics-lem}, it remains to prove that in the situation of Lemma \ref{main-rk2-cubics-lem} one has $\codim L_f\le 6$. This
is clear in cases (1) and (2). In case (3), it is easy to check that $\PP_f$ consists of $4$ elements:
$$(x_1,x_2), (y_1,y_2), (x_1,z_3), (y_1,z_2).$$
The corresponding intersection has codimension $6$.
\end{proof}

\section{Hypersurfaces of higher degree}

\subsection{Proof of Theorem C(ii)}

We use induction on $d\ge 1$.
The case $d=1$ is clear, so assume that $d\ge 2$ and the assertion holds for degrees $<d$.
Assume that $\dim\sum_{P\in \PP_f} P >d^2-1$ (otherwise we are done), and let $\{P_1,\ldots,P_n\}$ be a minimal
subset of $\PP_f$ such that
$$\dim \sum_{i=1}^n P_i>d^2-1.$$
Note that by minimality, $\dim \sum_{i=1}^{n-1} P_i\le d^2-1$, so 
$$\dim \sum_{i=1}^n P_i\le d^2+1.$$

\medskip

\noindent
{\bf Claim}. There are no nonzero homogeneous polynomials of degree $d-1$ in the ideal $(P_1)\cap\ldots\cap (P_n)$.

Indeed, suppose $g\in (P_1)\cap\ldots\cap (P_n)$ is such a polynomial. 
We have one of the two cases:

\noindent
{\bf Case 1}. $g=l_1\ldots l_k\cdot h$, where $\deg l_i=1$, $0\le k<d-2$, $\srk(h)\ge 2$.

\noindent
{\bf Case 2}. $g=l_1\ldots l_{d-1}$, where $\deg l_i=1$.

Let us consider Case 1 first.
Since each $(P_i)$ is a prime ideal, we should have a decomposition
$$\{1,\ldots,n\}=S_1\cup \ldots\cup S_k\cup S,$$
where $l_j\in P_i$ for all $i\in S_j$ and $h\in (P_i)$ for $i\in S$ (and $S=\emptyset$ if $\srk h>2$).

Let us fix $j$ such that $S_j\neq \emptyset$.
Then $f\mod (l_i)$ has slice rank $1$, hence $\dim\sum_{i\in S_j} P_i/(l_i)\le d$ (by Lemma \ref{rk1-lem}). In other words,
$$\dim \sum_{i\in S_j} P_i\le d+1.$$
On the other hand, assuming that $S\neq\emptyset$ and applying the induction hypothesis to $h$, we get
$$\dim \sum_{i\in S} P_i\le (d-1-k)^2+1.$$
Hence, we obtain
$$\dim \sum_{i=1}^n P_i\le k(d+1)+(d-1-k)^2+1\le d^2-1,$$
which is a contradiction.

Similarly, in Case 2 we get
$$\dim \sum_{i=1}^n P_i\le (d-1)(d+1)=d^2-1,$$
which is a contradiction. This proves the Claim.

Combining the Claim with Lemma \ref{ideals-intersection-lem}, we get the inclusion
$$f\in (P_1)\cap\ldots\cap (P_n)\sub (P_1+\ldots+P_n)^d.$$
Hence, $f$ is a pullback from a space of dimension $\le d^2+1$.
This finishes the proof. 

\subsection{Proof of Theorem D}



Let us dualize the recursive procedure described in Definition \ref{recursion-subspaces-def}. For a collection $\PP=(P_1,\ldots,P_s)$ of subspaces of $V^*$
we set $P^{(1)}=\cap_{i=1}^s P_i$, and for each minimal subset $J\sub [1,s]$ such that $\cap_{j\in J} P_j=P^{(1)}$, we set $P_J:=\sum_{j\in J} P_j$.
We denote by $\PP^{(1)}$ the collection of all subspaces $P_J$ of $V^*$ obtained in this way. Iterating this procedure we get collections of
subspaces $\PP^{(i)}$ for $i\ge 0$, where $\PP^{(0)}=\PP$. Let us also set $P^{(0)}=0$ and for $i\ge 0$,
$$P^{(i+1)}:=\cap_{P\in \PP^{(i)}} P.$$
Note that $P^{(i)}\sub P^{(i+1)}$.

\medskip

\noindent
{\bf Step 1}. If $\dim P_i\le r$ for all $i$ then $\dim P_J\le r^2$.
Indeed, let $a=\dim P^{(1)}$. Then $\dim P_i/P^{(1)}\le r-a$ and applying Lemma \ref{subspaces-intersection-lem} we
see that for every minimal subset $J$ with $\cap_{j\in J}P_j=P^{(1)}$, one has
$$\dim P_J=a+\dim P_J/P^{(1)}\le a+(r-a)+\frac{(r-a+1)^2}{4}\le r+\frac{(r+1)^2}{4}.$$
Since
$$\lfloor r+\frac{(r+1)^2}{4} \rfloor\le r^2$$
for $r\ge 2$, the assertion follows.

\medskip

\noindent
{\bf Step 2}. Suppose $f$ is a homogeneous polynomials such that $f\in (P_i)$ for $i=1,\ldots,s$. Let us prove by induction on $i\ge 0$ that
$$f\in (P^{(i)})+(P)^{i+1}$$
for any $P\in \PP^{(i)}$. Indeed, for $i=0$ this is true by assumption. Assume that $i>0$ and the assertion holds for $i-1$. 
Let us apply Lemma \ref{ideals-intersection-lem} to 
a collection of subspaces $\{Q_1,\ldots,Q_p\}\sub \PP^{(i-1)}$ such that $Q_1\cap\ldots\cap Q_p=P^{(i)}$, or rather to the corresponding
subspaces $\ov{Q}_i=Q_i/P^{(i)}$ of $V^*/P^{(i)}$. We get the inclusion of ideals
$$(\ov{Q}_1)^i\cap\ldots\cap (\ov{Q}_p)^i\sub (\sum \ov{Q}_j)^{i+1}$$
in the symmetric algebra of $V^*/P^{(i)}$. Let us
consider the polynomial $\ov{f}=f \mod (P^{(i)})$ in this algebra.
By assumption, $\ov{f}\in (\ov{Q}_j)^i$ for $j=1,\ldots,p$. Hence, we deduce that
$\ov{f}\in (\sum \ov{Q}_j)^{(i+1)}$, i.e.,
$$f\in (P^{(i)})+(\sum Q_j)^{(i+1)}.$$
Since every subspace in $\PP^{(i)}$ has form $\sum Q_j$, with $(Q_1,\ldots,Q_p)$ as above, this proves the induction step.

\medskip

\noindent
{\bf Step 3}. For $i=d$, since $f$ is homogeneous of degree $d$, the result of the previous step gives
$$f\in (P^{(d)}).$$
Recall that $P^{(d)}$ is the intersection of all subspaces in $\PP^{(d-1)}$. Iterating the result of Step 1, we see that the dimension of
any subspace in $\PP^{(d-1)}$, and hence of $P^{(d)}$, is $\le r^{2^{d-1}}$.
This ends the proof of Theorem D.

\begin{remark}\label{family-rk-rem}
Suppose we have an $s$-dimensional subspace $F$ of homogeneous polynomials of the same degree $d$, defined over $\k$, such that there exists a nonzero
$f\in F_{\ov{\k}}$ and subspace of linear forms of dimension $r$ over $\ov{\k}$ such that $f\in (L)$. One can ask how to produce an element $f_0\in F\setminus 0$ and a subspace
of linear forms $L_0$, both defined over $\k$, such that $f_0\in (L_0)$ and dimension of $L_0$ is $\le c(sr)$ (by Theorem A, we know that such an element exists).

Let $F_0\sub F$ denote the subspace spanned by all the Galois conjugates of $f$. Then $F_0$ is defined over $\k$. As $f_0$ we will take any nonzero element of $F_0$.

Since $\dim F_0\le \dim F\le s$, we can choose a set of elements of the Galois group $\si_1,\ldots,\si_s$, such that $(\si_1 f,\ldots,\si_s f)$ span $F_0$.
Hence, $f_0$ is a linear combination of $(\si_1 f,\ldots,\si_s f)$, and so,
$$f_0\in (\si_1 L+\ldots+\si_s L).$$
Now applying our algorithm from Theorem D for $f_0$, we find a subspace $L_0$ of dimension $\le c(sr)$ defined over $\k$, with $f_0\in (L_0)$.
\end{remark}

\end{document}